\documentclass[12pt]{amsart}
\usepackage{amssymb,amsmath,amscd,eucal,epsfig}
\usepackage{tikz}

 \def\Dj{\hBox{D\kern-.73em\raise.30ex\hBox{-} \raise-.30ex\hBox{}}}
 \def\dj{\hBox{d\kern-.33em\raise.80ex\hBox{-} \raise-.80ex\hBox{\kern-.40em}}}

\def\<{\langle}                     
\def\>{\rangle}                     

\theoremstyle{plain}
\newtheorem{theorem}{Theorem}[section]

\newtheorem{lemma}{Lemma}[section]

\newtheorem{remark}{Remark}[section]

\theoremstyle{definition}
\newtheorem{definition}{Definition}[section]

\numberwithin{equation}{section}

\begin{document}
\setcounter{page}{1}


\title[Radio Number of Stacked-Book Graph]{Bounds of the Radio Number of Stacked-Book Graph with Odd Paths}


\author[T.C Adefokun]{Tayo Charles Adefokun$^1$ }
\address{$^1$Department of Computer and Mathematical Sciences,
\newline \indent Crawford University,
\newline \indent Nigeria}
\email{tayoadefokun@crawforduniversity.edu.ng}

\author[D.O. Ajayi]{Deborah Olayide Ajayi$^2$}
\address{$^2$Department of Mathematics,
	\newline \indent University of Ibadan,
	\newline \indent Ibadan,
	\newline \indent Nigeria}
\email{olayide.ajayi@mail.ui.edu.ng; adelaideajayi@yahoo.com}






\begin{abstract} 
A Stacked-book graph $G_{m,n}$ is obtained from the Cartesian product of a star graph $S_m$ and a path $P_n$, where $m$ and $s$ are the orders of the star graph and the path respectively. Obtaining the radio number of a graph is a rigorous process, which is dependent on diameter of $G$ and positive difference of non-negative integer labels $f(u)$ and $f(v)$ assigned to any two $u,v$ in the vertex set $V(G)$ of $G$. This paper obtains tight upper and lower bounds of the radio number of $G_{m,n}$ where the path $P_n$ has an odd order. The case where $P_n$ has an even order has been investigated. 
\end{abstract}


\maketitle


\section{Introduction}
All graphs mentioned in this work are simple and undirected. The vertex and edge sets of a graph $G$ are designated as $V(G)$ and $E(G)$ respectively and $e=uv \in E(G)$ connects $u,v \in V(G)$ while $d(u,v)$ denotes the shortest distance between $u,v \in V(G)$. We represent the diameter of $G$ as $\rm{diam}(G)$. 

Radio labeling, which often aims to solve signal interference problems in a wireless network, was first suggested in 1980 by Hale \cite{Hale} and it is described as follows: Suppose that $f$ is a non negative integer function on $V(G)$ and that $f$ satisfies the radio labeling condition, $|f(u)-f(v)| \geq \rm{diam} (G)+1-d(u,v)$, for every pair $u,v \in V(G)$. The span$f$ of $f$ is $f_{max}(G)-f_{min}(G)$, where $f_{max}$ and $f_{min}$ are largest and lowest labels, respectively, assigned on $V(G)$ and the lowest value of span$f$ is the radio number, $rn(G)$, of $G$.  It is established that to obtain the radio numbers of graphs is hard. However, for certain graphs, the radio numbers have been obtained. Recent results on radio number include those on middle graph of path \cite{BD1}, trees, \cite{BD2} and edge-joint graphs \cite{NSS1}. Liu and Zhu \cite{LZ1} showed that  for path, $P_n$, $n \geq 3$, 
\begin{center}

$rn(P_n) = \left\{
\begin{array}{ll}
            2k(k-1)+1 &  \mbox{if} \;\; n=2k;   \\
             2k^2+2 &  \mbox{if} \; \; n=2k+1.
					
\end{array}
\right.$

\end{center}
 Liu and Zhu's results compliment those obtained by Chatrand et. al. in   \cite{CEHZ1} and \cite{CEZ1} about the same graph. Liu and Xie worked on square graphs. In \cite{LX2}, they obtained $rn(P_n^2)$ of square of path as follows:

\begin{center}

$rn(P^2_n) = \left\{
\begin{array}{ll}
            k^2+2 &  \mbox{if} \;\; n \equiv 1 (\rm{mod} \; 4), n \geq 9;   \\
             k^2+1 &  \mbox{if} \; \; \rm{otherwise}.
					
\end{array}
\right.$

\end{center}
Other results on squares of graphs include those obtained for $C_n^2$ in \cite{LX1}, where $C_n$ is a cycle of order $n$. On Cartesian products graphs, Jiang \cite{J1} solved the radio number problem for $(P_m \Box P_n)$, where for $m,n > 2$, and obtains $rn(P_m \Box P_n)=\frac{mn^2+nm^2-n}{2}-mn-m+2$, for $m$ odd and $n$ even. Saha and Panigrahi \cite{SP1} worked on Cartesian products of Cycles while Ajayi and Adefokun  in \cite{AA1} and \cite{AA2} probe on the radio number of the Cartesian product of path and star graph called the stacked-book graph $G=S_m \Box P_n$. They observed in \cite{AA1} that $rn(S_m \Box P_n) \leq n^2m+1$, a result the authors noted, citing a existing result in \cite{J1},  is not a tight bound.)In \cite{AA1}, they obtained improve the results in \cite{AA1}, for path $P_n$, where $n$ is even. 

In this paper, we investigate further on the radio number of stacked-book graphs, $S_m \Box P_n$, in the case where $n$ is odd and combined with \cite{AA2}, we improve the weak bounds obtained in \cite{AA1}.


\section{Preliminaries}
Let $S_m$ be a star of order $m \geq 3$ and let $v_1$ be the center vertex of $S_m$ and $v_2, v_3, \cdots, v_m$ are adjacent to $v_1$ and let $P_n$ be a path containing $n$ vertices starting from $u_1$ to $u_n$. Furthermore, $P=u\stackrel{a}{\longrightarrow}v\stackrel{b}{\longrightarrow}w$ represents a path of length $a+b$, for which $d(u,v)=a$ and $d(u,w)=b$. If a stacked-book graph is obtained from the Cartesian product $G_{m,n}=S_m \Box P_n$ of $S_m$ and $P_n$, then $V(G_{m,n})$ is the Cartesian product of $V(S_m)$ and $V(P_n)$, for which if $u_iv_j \in V(G_{m,n})$, then $u_i\in V(S_m)$, $v_j \in E(P_n)$, while, if $u_iv_j u_kv_i$ forms an edge in $E(G_{m,n})$, then $u_i=u_k$ and $v_jv_l \in E(P_n)$ or $v_j=v_l$ and $u_iu_k \in E(S_m)$.  

Some of the following are adopted from \cite{AA2}.

\begin{definition}
Where it is convenient, we write $u_iv_j$ as $u_{ij}$ and therefore the edge $u_iv_j u_kv_l$ as $u_{ij}u_{kl}$.  
\end{definition} 

\begin{remark}
Stacked-book graph $G_{m,n}$ contains $n$ number of $S_m$ stars as contained in the set $\left\{S_{m(i)}:i \leq 0 \leq n\right\}$.
\end{remark}
\begin{definition}
For $G_{m,n}=S_m \Box P_n$, $V_{(i)} \subset V(G_{m,n})$ is the set of vertices on $S_{m(i)}$ stated as $V_{(i)}={u_1v_i, u_2v_i, \cdots, u_mv_i}$. 
\end{definition}
\begin{remark}
We must mention that $u_1v_i$ in the set in the last definition is the center vertex of $S_m(i)$.
\end{remark}
\begin{definition}
Let $G_{m,n}=S_m \Box P_n$, $n$ odd, the pair ${S_{m(i)}}, S_{m(i+\frac{n-1}{2})}$ is a subgraph $G''(i)\subseteq G_{m,n}$, which is induced by $V_i$ and $V_{i+\frac{n-1}{2}}$, where $i \notin \lbrace 1, \frac{n+1}{2},n\rbrace$.
\end{definition}
\begin{remark}
It can be seen that with $n$ odd, every $G_{m,n}$ contains $\frac{n-2}{3}$ number of $G''(i)$ subgraphs and the diameter $diam(G''(i))$ of $G''(i)$ is $\frac{n+3}{3}$.
\end{remark}

Next, we introduce the following definitions:

\begin{definition}
Let $G_{m,n}=S_m \Box P_n$. Then, $\bar{G}_{m,n} \subseteq G_{m,n}$ is a subgraph of $G_{m,n}$ induced by the stars ${S_{m(1)}, S_{m(\frac{n+1}{2})}, S_n}$.
\end{definition}
We now define a class of paths $P'(i)$.

\begin{definition}
Let $\left\{P'(t)\right\}^m_{t=1}$ be a class of paths in $G_{m,n}$, where $P'(t):=v_{j(1)}\stackrel{\alpha}{\longrightarrow}v_{k(\frac{n+1}{2})}\stackrel{\beta}{\longrightarrow}v_{l(n)}$, such that $j\neq k \neq l$, $v_{j(1)} \in V_{(i)}, v_{k(\frac{n+1}{2})} \in V_{(\frac{n+1}{2})}$ and $v_{l(n)} \in V_{(n)}$ and $1 \leq j, k, l \leq m$.
\end{definition}

It can be verified that $\left\{P'(t)\right\}^m_{t=1}$ contains two other sub-classes defined without loss of generality as follows:

$P'_1(t)=\lbrace v_{1(1)}\stackrel{\frac{n+1}{2}}{\longrightarrow}v_{3(\frac{n+1}{2})}\stackrel{\frac{n+3}{2}}{\longrightarrow}v_{2(n)},
v_{2(1)}\stackrel{\frac{n+1}{2}}{\longrightarrow}v_{1(\frac{n+1}{2})}\stackrel{\frac{n+1}{2}}{\longrightarrow}v_{3(n)},
v_{3(1)}\stackrel{\frac{n+3}{2}}{\longrightarrow}v_{2(\frac{n+1}{2})} \stackrel{\frac{n+1}{2}}{\longrightarrow}v_{1(n)} \rbrace$

$P'_2(t)=v_{a(1)}\stackrel{\frac{n+3}{2}}{\longrightarrow}v_{b(\frac{n+1}{2})}\stackrel{\frac{n+3}{2}}{\longrightarrow}v_{c(n)}$, $a \neq b \neq c$, $4 \leq a,b,c \leq m$. Clearly, $|P'_1(t)|=3$ and $|P'_2(t)|=m-2$.
       
\section{Results}

In the next results, we establish a lower bound of the radio number $rn(G_{m,n})$ of a stacked-book graph $G_{m,n}$.

\begin{lemma}\label{lem1}
Let $f$ the radio labeling function on $G_{m,n}$, $n$ odd, and let $V_{(\frac{n+1}{2})}=\left\{v_{1(\frac{n+1}{2})}, v_{2(\frac{n+1}{2})}, v_{3(\frac{n+1}{2})}, \left\{v_{d(\frac{n+1}{2}):4 \leq d \leq m)}\right\} \right\}$ be the vertex set of the mid vertices in $P(t) \subseteq \left\{P'(t)\right\}^m_{t=1}$.  Now, let $v\in V_{\frac{n+1}{2}}$ be some vertex in $V_{\frac{n+1}{2}}$. If $f(v)$ is $f_{max}$ on $V(P(t))$, then

\begin{center}

$f(v) = \left\{
\begin{array}{ll}
            \frac{n+5}{2} &  \mbox{if} \;\; v\in \left\{v_{1(\frac{n+1}{2})}, v_{2(\frac{n+1}{2})} , v_{3(\frac{n+1}{2})} \right\};   \\
             \frac{n+3}{2} &   \rm{otherwise}.
					
\end{array}
\right.$

\end{center}   
\end{lemma}
\begin{proof}
Since $P(t) \subset G_{m,n}$, then, radio labeling of any vertex on $V(P(t))$ is based on $diam(G_{m,n})$ and for $u,v \in V(P(t))$, $d(u,v)=k$, where $k$ is the distance between $u$ and $v$ in $G_{m,n}$. We consider the three paths in $P'_1(t)$ next.

{\bf{Case 1}}(a). For $P'_1(1):=v_{1(1)}\stackrel{\frac{n+1}{2}}{\longrightarrow}v_{3(\frac{n+1}{2})}\stackrel{\frac{n+3}{2}}{\longrightarrow}v_{2(n)}$, let $f(v_{1(1)})=0$. Now $d(v_{1(1)},v_{2(n)})=n$. Therefore $f(v_{2(n)}) \geq f(v_{1(1)})+diam(G_{m,n})+1-n$=2. Also, $d(v_{2(n)},v_{3(\frac{n+1}{2})})=\frac{n+3}{2}$ and thus, $f(v_{3(\frac{n+1}{2})}) \geq f(v_{2(n)})+dim(G_{m,n})+1-\frac{n+3}{2} \geq \frac{n+5}{2}$. (It should be note that if we set $f(v_{2(n)})=0,$ then, $f(v_{3(\frac{n+1}{2})}) \geq \frac{n+7}{2}$.)

{\bf{Case 1}}(b). For $P'_1(2):=v_{2(1)}\stackrel{\frac{n+1}{2}}{\longrightarrow}v_{1(\frac{n+1}{2})}\stackrel{\frac{n+1}{2}}{\longrightarrow}v_{3(n)}$, let $f(v_{2(1)})=0$, then $d(v_{2(1)},v_{3(n)})=n+1$ and thus, $f(v_3(n)) \geq n+2-(n+1)=1$. Likewise, $d(v_{3(n)},v_{1(\frac{n+1}{2})})=\frac{n+2}{2}$ and therefore, $f(v_{3(n)}) \geq n+3-(\frac{n+1}{2})=\frac{n+5}{2}$.

\begin{center}
	\pgfdeclarelayer{nodelayer}
	\pgfdeclarelayer{edgelayer}
	\pgfsetlayers{nodelayer,edgelayer}
	\begin{tikzpicture}
	\begin{pgfonlayer}{nodelayer}

	\node [minimum size=0cm,draw,circle] (0) at (2,0.5) {\tiny{}};
	\node [minimum size=0cm,draw,circle] (1) at (5,0.5) {\tiny{}};
	\node [minimum size=0cm,draw,circle] (2) at (8,0.5) {\tiny{}};
	\node [minimum size=0cm,draw,circle] (3) at (11,0.5) {\tiny{}};
	\node [minimum size=0cm,draw,circle] (4) at (14,0.5) {\tiny{1}};
	
	\node [minimum size=0cm,draw,circle] (5) at (0,1) {\tiny{}};
	\node [minimum size=0cm,draw,circle] (6) at (3,1) {\tiny{}};
	\node [minimum size=0cm,draw,circle] (7) at (6,1) {\tiny{}};
	\node [minimum size=0cm,draw,circle] (8) at (9,1) {\tiny{}};
	\node [minimum size=0cm,draw,circle] (9) at (12,1) {\tiny{}};
	
	\node [minimum size=0cm,draw,circle,red] (10) at (1,2) {\tiny{0}};
	\node [minimum size=0cm,draw,circle] (11) at (4,2) {\tiny{}};
	\node [minimum size=0cm,draw,circle] (12) at (7,2) {\tiny{5}};
	\node [minimum size=0cm,draw,circle] (13) at (10,2) {\tiny{}};
	\node [minimum size=0cm,draw,circle] (14) at (13,2) {\tiny{}};
	
	\node [minimum size=0cm,draw,circle] (15) at (0,3) {\tiny{0}};
	\node [minimum size=0cm,draw,circle] (16) at (3,3) {\tiny{}};
	\node [minimum size=0cm,draw,circle] (17) at (6,3) {\tiny{}};
	\node [minimum size=0cm,draw,circle] (18) at (9,3) {\tiny{}};
	\node [minimum size=0cm,draw,circle,red] (19) at (12,3) {\tiny{2}};
	
	\node [minimum size=0cm,draw,circle] (20) at (2,3.5) {\tiny{}};
	\node [minimum size=0cm,draw,circle] (21) at (5,3.5) {\tiny{}};
	\node [minimum size=0cm,draw,circle,red] (22) at (8,3.5) {\tiny{5}};
	\node [minimum size=0cm,draw,circle] (23) at (11,3.5) {\tiny{}};
	\node [minimum size=0cm,draw,circle] (24) at (14,3.5) {\tiny{}};

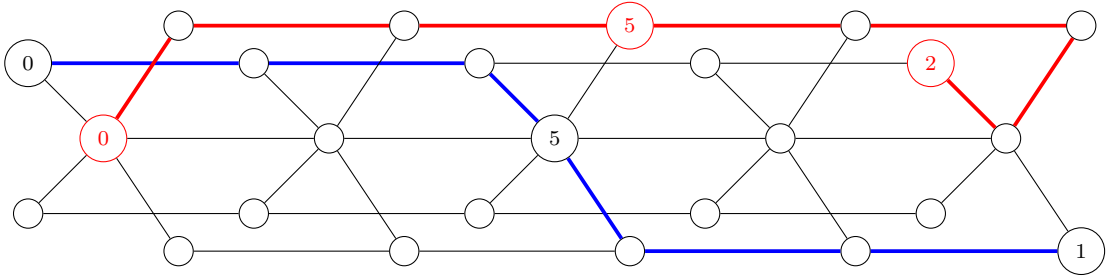
	\node [minimum size=0] (25) at (6,-1) {Figure 1. Illustration of Case 1(a) and (b) in a $G_{5,5}$ Stacked-book Graph };
	
	\end{pgfonlayer}
	\begin{pgfonlayer}{edgelayer}
	\draw [thin=1.00] (0) to (1);
	\draw [thin=1.00] (1) to (2);
      \draw [line width=0.5mm, blue] (2) to (3);
	\draw [line width=0.5mm, blue] (3) to (4);
	
	\draw [thin=1.00] (5) to (6);
	\draw [thin=1.00] (6) to (7);
	\draw [thin=1.00] (7) to (8);
	\draw [thin=1.00] (8) to (9);
	
	\draw [thin=1.00] (10) to (11);
	\draw [thin=1.00] (11) to (12);
	\draw [thin=1.00] (12) to (13);
	\draw [thin=1.00] (13) to (14);

	\draw [line width=0.5mm, blue] (15) to (16);
	\draw [line width=0.5mm, blue] (16) to (17);
	\draw [thin=1.00] (17) to (18); 
	\draw [thin=1.00] (18) to (19);
	
	\draw [line width=0.5mm, red] (20) to (21);
	\draw [line width=0.5mm, red] (21) to (22);
	\draw [line width=0.5mm, red] (22) to (23);
	\draw [line width=0.5mm, red] (23) to (24);
	
	\draw [thin=1.00] (10) to (5);
	\draw [thin=1.00] (10) to (15);
	\draw [thin=1.00] (10) to (0);
	\draw [line width=0.5mm, red] (10) to (20);
	
	\draw [thin=1.00] (11) to (16);
	\draw [thin=1.00] (11) to (6);
	\draw [thin=1.00] (11) to (1);
	\draw [thin=1.00] (11) to (21);
	
	\draw [line width=0.5mm, blue] (12) to (17);
	\draw [thin=1.00] (12) to (7);
	\draw [line width=0.5mm, blue] (12) to (2);
	\draw [thin=1.00] (12) to (22);
	
	\draw [thin=1.00] (13) to (18);
	\draw [thin=1.00] (13) to (8);
	\draw [thin=1.00] (13) to (3);
	\draw [thin=1.00] (13) to (23);
	
	\draw [line width=0.5mm, red] (14) to (19);
	\draw [thin=1.00] (14) to (9);
	\draw [thin=1.00] (14) to (4);
	\draw [line width=0.5mm, red] (14) to (24);

	\end{pgfonlayer}
	\end{tikzpicture}
	
\end{center}

{\bf{Case 1}}(c). Now for $P'_1(3):=v_{3(1)}\stackrel{\frac{n+3}{2}}{\longrightarrow}v_{2(\frac{n+1}{2})}\stackrel{\frac{n+1}{2}}{\longrightarrow}v_{1(n)}$, we assume $f(v_{1(n)})=0$. Also, $d(v_{3(1)},v_1(n))=n+1$ in $G_{m,n}$. Thus, $f(v_{3(1)}) \geq 2$ and since $d(v_{3(1)},v_{2{\frac{n+1}{2}}}) \geq \frac{n+3}{2}$, then, $f(v_{2(\frac{n+1}{2})}) \geq 2+n+2-(\frac{n+3}{2}) \geq \frac{n+5}{2}$.

Next we consider the paths in $P'_2(t)$.

{\bf{Case 2}}. Every path in $P'_2(t)$ are geometrically similar and are of the form $P'_2(4)=v_{a(1)}\stackrel{\frac{n+3}{2}}{\longrightarrow}v_{b(\frac{n+1}{2})}\stackrel{\frac{n+3}{2}}{\longrightarrow}v_{n(n)}$, such that $d(v_{a(1)},v_{c(n)})=n+1$ and $d(v_{c(n)},v_b(\frac{n+1}{2}))=\frac{n+3}{2}$, in $G_{m,n}$ and for all $a\neq b \neq c \neq m$, without loss of generality. Thus, suppose that $f(v_{a(1)}) = 0$, then $f(v_c(n)) \geq 1$ and $f(v_{b(\frac{n+1}{2})}) \geq \frac{n+3}{2}$.

\end{proof}

\begin{center}
	\pgfdeclarelayer{nodelayer}
	\pgfdeclarelayer{edgelayer}
	\pgfsetlayers{nodelayer,edgelayer}
	\begin{tikzpicture}
	\begin{pgfonlayer}{nodelayer}

	\node [minimum size=0cm,draw,circle,red] (0) at (2,0.5) {\tiny{0}};
	\node [minimum size=0cm,draw,circle] (1) at (5,0.5) {\tiny{}};
	\node [minimum size=0cm,draw,circle] (2) at (8,0.5) {\tiny{}};
	\node [minimum size=0cm,draw,circle] (3) at (11,0.5) {\tiny{}};
	\node [minimum size=0cm,draw,circle] (4) at (14,0.5) {\tiny{}};
	
	\node [minimum size=0cm,draw,circle] (5) at (0,1) {\tiny{0}};
	\node [minimum size=0cm,draw,circle] (6) at (3,1) {\tiny{}};
	\node [minimum size=0cm,draw,circle] (7) at (6,1) {\tiny{}};
	\node [minimum size=0cm,draw,circle] (8) at (9,1) {\tiny{}};
	\node [minimum size=0cm,draw,circle,red] (9) at (12,1) {\tiny{1}};
	
	\node [minimum size=0cm,draw,circle] (10) at (1,2) {\tiny{}};
	\node [minimum size=0cm,draw,circle] (11) at (4,2) {\tiny{}};
	\node [minimum size=0cm,draw,circle] (12) at (7,2) {\tiny{5}};
	\node [minimum size=0cm,draw,circle] (13) at (10,2) {\tiny{}};
	\node [minimum size=0cm,draw,circle] (14) at (13,2) {\tiny{}};
	
	\node [minimum size=0cm,draw,circle] (15) at (0,3) {\tiny{}};
	\node [minimum size=0cm,draw,circle] (16) at (3,3) {\tiny{}};
	\node [minimum size=0cm,draw,circle, red] (17) at (6,3) {\tiny{4}};
	\node [minimum size=0cm,draw,circle] (18) at (9,3) {\tiny{}};
	\node [minimum size=0cm,draw,circle] (19) at (12,3) {\tiny{}};
	
	\node [minimum size=0cm,draw,circle] (20) at (2,3.5) {\tiny{}};
	\node [minimum size=0cm,draw,circle] (21) at (5,3.5) {\tiny{}};
	\node [minimum size=0cm,draw,circle] (22) at (8,3.5) {\tiny{}};
	\node [minimum size=0cm,draw,circle] (23) at (11,3.5) {\tiny{}};
	\node [minimum size=0cm,draw,circle] (24) at (14,3.5) {\tiny{1}};

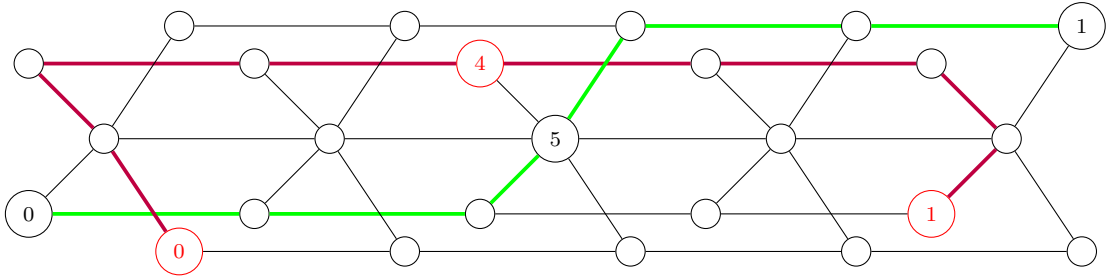
	\node [minimum size=0] (25) at (6,-1) {Figure 1. Illustration of Case 1(c) and Case 2 in a $G_{5,5}$ Stacked-book Graph };
	
	\end{pgfonlayer}
	\begin{pgfonlayer}{edgelayer}
	\draw [thin=1.00] (0) to (1);
	\draw [thin=1.00] (1) to (2);
	\draw [thin=1.00] (2) to (3);
	\draw [thin=1.00] (3) to (4);
	
	\draw [line width=0.5mm, green] (5) to (6);
	\draw [line width=0.5mm, green] (6) to (7);
	\draw [thin=1.00] (7) to (8);
	\draw [thin=1.00] (8) to (9);
	
	\draw [thin=1.00] (10) to (11);
	\draw [thin=1.00] (11) to (12);
	\draw [thin=1.00] (12) to (13);
	\draw [thin=1.00] (13) to (14);
	
	\draw [line width=0.5mm, purple] (15) to (16);
	\draw [line width=0.5mm, purple] (16) to (17);
	\draw [line width=0.5mm, purple] (17) to (18); 
	\draw [line width=0.5mm, purple] (18) to (19);
	
	\draw [thin=1.00] (20) to (21);
	\draw [thin=1.00] (21) to (22);
	\draw [line width=0.5mm, green] (22) to (23);
	\draw [line width=0.5mm, green] (23) to (24);
	
	\draw [thin=1.00] (10) to (5);
	\draw [line width=0.5mm, purple] (10) to (15);
	\draw [line width=0.5mm, purple] (10) to (0);
	\draw [thin=1.00] (10) to (20);
	
	\draw [thin=1.00] (11) to (16);
	\draw [thin=1.00] (11) to (6);
	\draw [thin=1.00] (11) to (1);
	\draw [thin=1.00] (11) to (21);
	
	\draw [thin=1.00] (12) to (17);
	\draw [line width=0.5mm, green] (12) to (7);
	\draw [thin=1.00] (12) to (2);
	\draw [line width = 0.5mm, green] (12) to (22);
	
	\draw [thin=1.00] (13) to (18);
	\draw [thin=1.00] (13) to (8);
	\draw [thin=1.00] (13) to (3);
	\draw [thin=1.00] (13) to (23);
	
	\draw [line width=0.5mm, purple] (14) to (19);
	\draw [line width=0.5mm, purple] (14) to (9);
	\draw [thin=1.00] (14) to (4);
	\draw [thin=1.00] (14) to (24);

	\end{pgfonlayer}
	\end{tikzpicture}
	
\end{center}

\begin{remark}\label{rem1}
In (a) and (c) of Case 1, if the respective center vertices $v_{1(1)}$ and $v_{1(n)}$ of stars $S_{(1)}$ and $S_{(n)}$ are labeled $f(v_{1(1)})=f(v_{1(n)})=0$, the radio labels on the mid vertices of their paths would be at least $\frac{n+7}{2}$. 
\end{remark}

\begin{remark}\label{rem2}
For the $m$ paths in $\left\{P'{t}\right\}^m_{t=1}$, the sum of  all the radio labels on the center vertices ($span(f)$ of $f$ on $P'(t)$ is $3(\frac{n+5}{2})+(m-3)(\frac{n+3}{2})=\frac{1}{2}(mn+3m+6)$.
\end{remark}

Next, we obtained a lower bound for $\left\{P(t)\right\}^m_{t=1}$.

\begin{remark}\label{rem3}
From Remark \ref{rem2}, we notice that for optimum labeling of the three vertices on each of the paths in $\left\{P(t)\right\}^m_{t=1}$, the end vertex, which closest to the mid vertex is most suitable to be labeled first. These are $v_{1(1)} \in P'_1(1), v_1(n) \in P'_1(3)$ and any end vertex in the remaining paths. We refer to each of these ends vertices as initial label vertex.
\end{remark}

\begin{lemma}\label{lem2}
Let $G(*)$ be a subgraph of $G_{m,n}$, induced by all the end point vertices and the midpoint vertices of $\left\{P'_1(t)\right\}^m_{t=1}$ i.e $S_{m(1)}, S_{m(\frac{n+1}{2})}, S_{m(n)}$. Then $rn(G(*)) \geq \frac{1}{2}(2mn+4m-n+5)$ in $G_{m,n}$.   
\end{lemma}

\begin{proof}
Let $v_1$ and $v_2$ be center vertices on $S_{m(1)}$ and $S_{m(n)}$ respectively. There exist vertices $u_{\alpha},u_{\beta} \in S_{m(\frac{n+1}{2})}$, $\alpha \neq \beta$, $u_{\alpha}, u_{\beta}$ not center vertices of $S_{m(\frac{n+1}{2})}$ such that $d(v_1,u_{\alpha})=d(v_2,u_{\beta})=\frac{n+1}{2}$. Also, there exists a subset $A=\left\{\omega_r \right\}$ in $S_{m(1)}$, (or $S_{m(n)}$) such that $|A|=m-3$,  and $B=\left\{x_s\right\}$ in $S_{m(\frac{n+1}{2})}$, $|B|=m-1$, such that for $r \neq s$, $d(\omega_r,x_s)=\frac{n+3}{2}$. Now, the sum of $span(f)$ of $f$ for all the pair ($\omega_r , x_s$) will be  $(m-1)(\frac{n+1}{2})= \frac{1}{2}(mn+m-n-1)$ and thus, 

\begin{eqnarray}
rn(G(*)) & \geq & \frac{1}{2}(mn+m-n-1)+ \frac{1}{2} mn+3m+6] \nonumber \\
& \geq & \frac{1}{2}(2mn+4m-n+5). \nonumber
\end{eqnarray}

\end{proof}
We extend the result in Lemma \ref{lem2} in other to obtain a lower bound for the radio number of stacked book graph $G_{m,n}$, with off $n \geq 5$

\begin{definition}
Let $G_{m,n}$ be a stacked-book graph with odd $n$, $n \geq 5$, and $m\geq 4$. Also, let $G(*)$ as defined earlier. We define subgraph $(G(**))$ of $G_{m,n}$ as $G(**)=G_{m,n}\backslash G(*)$.
\end{definition}

\begin{remark}\label{rem4}
We can see that $G(**)$ is a subgraph of $G_{m,n}$, induced by $\left\{S_{m(i)}\right\}^{n-1}_{i=2}$, $i \neq \frac{n+1}{2}$.
\end{remark}

\begin{definition}
Let $G''(t) \subseteq G(**)$ be a subgraph of $G_{m,n}$, such that $G''(t)$ is induced by $S_{m(t)}$ and $S_{m(t+\frac{n-1}{2})}$.
\end{definition}

\begin{remark}\label{rem5}
It should be noted that $G(**) \subset G_{m,n}$ contains exactly $\frac{n-3}{2}$ $G''(t)$ subgraphs.
\end{remark}

\begin{remark}\label{rem6}
Let $G''(t)$ induced by $S_{m(t)}$ and $S_{m(t+\frac{n-1}{2})}$ and let $V(S_{m(t)})= \lbrace u_1,u_2, \cdots, u_m \rbrace $ and $V(S_{m(t+\frac{n-1}{2})})$ $=\left\{v_1,v_2, \cdots, v_m\right\}$ be the vertex sets of $S_{m(t)}$ and $S_{m(t+\frac{n-1}{2})}$ where $u_1$ and $v_1$ are the respective center vertices. It can be seen that, $d(u_i,v_j)\in \left\{\frac{n+1}{2}, \frac{n+3}{2}\right\}$, where $i \neq j$.
\end{remark}

\begin{remark} \label{rem7}
For $i \neq j$, $d(u_1, v_j)=d(u_j, v_1)=\frac{n+1}{2}$ and for $i \neq j$, $i,j \neq 1$, $d(u_i,v_j)=\frac{n+3}{2}$.
\end{remark}

Now we obtain a lower bound value for the radio number labeling of $G''(t)$ in $G_{m,n}$.

\begin{lemma} \label{lem3}
Let $G''(t) \subset G_{m,n}$, with $m \geq 4$ and $n \geq 5$, $n$ odd, be a subgraph of $G_{m,n}$. Then $rn(G''(t)) \geq mn+m-\frac{1}{2}(n-3)$. 
\end{lemma}

\begin{proof}
Let $u_1$ and $v_1$ be center vertices of $S_{m(t)}$ and $S_{m(t+\frac{n-1}{2})}$. By Remark \ref{rem7} above, $d(u_1,v_i)=d(u_i,v_1)= \frac{n+1}{2}$, is the shortest distance between the center vertex of a star in $G''(t)$ and a non-center vertex in the other star in $G''(t)$. It is optimal, therefore to label the center vertices as $f_{min}$ and $f_{max}$. Now, without loss of generality, set $f_{min}=f(v_1)=0$. Since $d(v_1,u_i)=\frac{n+1}{2}$, $i \in \left\{2,3, \cdots, m\right\}$. Therefore $f(u_i) \geq f(v_1) +diam(G_{m,n})+1-d(u_i,v_1)$. Let $i=2$. Thus,

\begin{eqnarray}
f(u_2) & \geq & 0+n+2-\frac{n+1}{2} \nonumber \\
& \geq & \frac{n+3}{2}. \nonumber
\end{eqnarray}

Now $d(u_2,v_3)= \frac{n+3}{2}$ and therefore,

\begin{eqnarray}
f(v_3) & \geq & \frac{n+3}{2}+n+2-\frac{n+3}{2} \nonumber \\
& \geq & \frac{n+3}{2}+\frac{n+1}{2}. \nonumber
\end{eqnarray}
Also, for $d(v_3,u_4)=\frac{n+3}{2}$, $f(v_4) \geq \frac{n+3}{2}+2(\frac{n+1}{2})$. By continuing the iteration, we have $f(v_m) \geq \frac{n+3}{2}+2m-3(\frac{n+1}{2})$. Lastly, $f_{max}=f(u_1) \geq 2(\frac{n+3}{2})+(2m-3)(\frac{n+1}{2})$ $=mn+m-\frac{1}{2}(n-3)$.
\end{proof}

Next we extend the last result to obtain a lower bound for $G(**)$.

\begin{lemma}\label{lem4}
For $G(**) \subset G_{m,n}$, $rn(G(**)) \geq \frac{1}{2}(mn^2-2mn-3m+2n-12)$.
\end{lemma}

\begin{proof}
From \ref{lem3}, the $span(f)$ of $f$ on $G''(t)=mn+m-\frac{1}{2}(n-3)$. For $G''(t)$, $f_{max}=mn+m-\frac{1}{2}(n-3)$. Let $t=2$ and let $u_1=v_{m(2)} \in S_{m(2)}$ and $v'_1=v_{m(2+\frac{n+1}{2})} \in S_{m(2+\frac{n+1}{2})}$, be center vertices of $S_{m(2)}$ and $S_{m(2+\frac{n+1}{2})}$. Now, $d(u_i,v'_1) = \frac{n+1}{2}$. Thus, $f(v'_1) \geq f(u_1)+n+2-\frac{n+1}{2}=f(u_1)+\frac{n+3}{2}$. This implies that for $G''(3)$, induced by $S_{m(3)}$ and $S_{m(2+\frac{n+1}{2})}$, $f_{min}=f(u_1)+\frac{n+3}{2}$, and $f_{max}=f(v''_1)$, where $v''$ is the center vertex of $S_{m(3)}$. From the precedure in Lemma \ref{lem3}, there are $\frac{n-3}{2}$ $G''(t)$ subgraphs in $G_{m,n}$. Therefore, $f_{max}$ of $G(**)$ is $f(v^{(k)}_1) \in S_{(m(\frac{n-1}{2}))}$, where $f(v_1^{(k)})$ is the center vertex of $S_{m(\frac{n-1}{2})}$.  Following the iteration,

\begin{eqnarray}
f(v^{(k)}_1) & \geq & \frac{n-3}{2}[mn+m-\frac{1}{2}(n-3)]+\frac{n-5}{2} \left( \frac{n+3}{2}\right) \nonumber \\
& \geq & \frac{1}{2}(mn^2-2mn-3m+2n-12). \nonumber
\end{eqnarray}
 
\end{proof}

Now we establish a lower bound for the radio number of $G_{m,n}$.

\begin{lemma}\label{lem5}
Let $G_{m,n}$ be a stacked-book graph with $m \geq 4$ and $n \geq 5$. Then, $rn(G_{m,n}) \geq \frac{mn^2+m+2n-4}{2}$.
\end{lemma}

\begin{proof}
From Lemmas \ref{lem3} and \ref{lem4}, $rn(G(*)) \geq m(n+2)- \frac{1}{2}(n-5)$ and $rn(G(**)) \geq \frac{1}{2}(mn^2-2mn-3n+2n-12)$. Now, since $G_{m,n}=G(*) \cup G(**)$, suppose that $u_1$ is the center vertex of $S_{m(\frac{n-1}{2})}$ and $v_1 \in S_{m(n)}$ is the center vertex of $S_{m(n)}$. Clearly $d(u_1,v_i)= \frac{n+1}{2}$. Now, $f(u_1)=f_{max}$ of $G(**)$ and $f(u_1) \geq \frac{1}{2}(mn^2-2mn-3m+2n-12)$. Therefore, 

\begin{eqnarray}
f(v_1) & \geq & f(u_i)+n+2-\frac{(n+1)}{2} \nonumber \\
& \geq & \frac{1}{2}(mn^2-2mn-3m+2n-12)+n+2-\frac{(n+1)}{2}  \nonumber\\
& \geq & \frac{1}{2}(mn^2-2mn-3m+n-13)+n+2.\nonumber
\end{eqnarray}
For $G(*)$, set $f(v_1)= f_{min}$. Thus, $rn(G_{m,n}) \geq f(v_1)+rn(G(*))$ and hence,
\begin{eqnarray}
rn(G_{m,n}) & \geq &\frac{1}{2}(mn^2-2mn-3m+n-13)+n+2+m(n+2)-\frac{1}{2}(n-5)  \nonumber \\
& \geq & \frac{mn^2+m+2n-4}{2}. \nonumber
\end{eqnarray}

\end{proof}
Next, we investigate the upper bound of a stacked-book graph. The technique involves manual radio labeling of subgraphs $G(*)$ and $G(**)$ and merging the results.

\begin{lemma}\label{lem6}
Let $G(**) \subset G_{m,n}$, with $n$-odd. Then, $rn(G(**)) \leq \frac{1}{2}(mn^2-2mn+2n-3m-12)$. 
\end{lemma}
\begin{proof}
From earlier definition, if $n$ is odd, then, $G_{m,n}=G(*) \cup G(**)$, where $G(**)$ contains $\frac{n-3}{2}$ $G''(t)$ graphs. Let $G''(\frac{n-1}{2})$ be induced by $S_{m(\frac{n-1}{2})}$ and $S_{m(n-1)}$, $n$-odd. Let $V(S_{m(\frac{n+1}{2})})=\left\{v_{\frac{n-1}{1}(i)}\right\}^m_{i=1}$, $V(S_{m(n-1)})=\left\{v_{n-1}\right\}^m_{i=1}$, where $v_{\frac{n-1}{2}(1)}$, $v_{n-1(1)}$ are center vertices and $d(v_{\frac{n-1}{2}(j)},v_{n-1(j)})=\frac{n-1}{2}$, for all $1 \leq j \leq m$, $d(v_{\frac{n-1}{2}(1)}),v_{n-1(j)}=d(v_{n-1(1)},v_{\frac{n-1}{2}(j)})=\frac{n+1}{2}$ and $d(v_{\frac{n-1}{2}(j)},v_{n-1(k)})=\frac{n+3}{2}$. Now, let $f(v_{\frac{n-1}{2}(1)})=0$. Since $d(v_{\frac{n-1}{2}(1)}, v_{n-1(2)})= \frac{n+1}{2}$, then set $f(v_{n-1(k)})=\frac{n+3}{2}$. Let $f(v_{\frac{n-1}{2}(1)})=0$. Since $d(v_{\frac{n-1}{2}(1),v_{n-1(2)}})=\frac{n+1}{2}$, then set $f(v_{n-1(2)})=\frac{n+3}{2}$, $d(v_{n-1(2)},v_{\frac{n-1}{2}(3)})=\frac{n+3}{2}$ and thus, $f(v_{n-1}(4))=\frac{n+3}{2}+2\frac{n+1}{2}$. Thus, by continuously alternating the process, it gets to the case where $d(v_{\frac{n-1}{2}(m)},d(v_{n-1(m-1)}))=\frac{n+3}{2}$. Thus, $f(\frac{n-1}{2}(m))=\frac{n+3}{2}+m-2(\frac{n+1}{2})$, and since $d(v_{\frac{n-1}{2}(m)},v_{n-1(3)})=\frac{n+3}{2}$, $f(v_{n-1(3)})=\frac{n+3}{2}+(m-1)(\frac{n+1}{2})$, $f(v_{\frac{n-1}{2}(2)})=\frac{n+3}{2}+m(\frac{n+1}{2})$. Depending on the size of $m$, the labeling continues until $f(v_{\frac{n-1}{2}})=\frac{n+3}{2}+2m-3\frac{(n+3)}{2}+2(2m-3)(\frac{n+1}{2})$ is attained and finally, $d(v_{\frac{n-1}{2}(m-1)},v_{n-1(1)})=\frac{n+1}{2}$ and thus, $f(v_{n-1(1)})=\frac{2(n+3)}{2}+2m-3+\frac{n+1}{2}$. (By following the same argument, it is easy to obtain similar result for $m$-even.) Now, $d(v_{n-1(1)},v_{\frac{n-3}{2}(1)})=\frac{n+1}{2}$, where $v_{\frac{n-3}{2}(1)}$ is the center vertex of $G''_{m(\frac{n-3}{2})}$. Therefore, $f_{min}(G''(\frac{n-3}{2}))=f(v_{\frac{n-3}{2}(1)})=f(v_{n-1(1)})+n+2-\frac{n+2}{2}=$$f(v_{n-1(1)})+\frac{n+3}{2}=\frac{3(n+3)}{2}+2m-3(\frac{n+1}{2})$, and $f_{max}(G''(\frac{n-3}{2})=f(v_{n-2(1)})=f(v_{\frac{n-3}{2}(1)})+\frac{2(n+3)}{2}+\frac{(2m-3)(n+1)}{2}=$$\frac{5(n+3)}{2}+2(2m-3)\frac{(n+1)}{2}$, which is $f_{max}(G''(\frac{n-3}{2}))$. Now, the process is extended to $G''(2)$, for which $f(\frac{n+3}{2})=\frac{(n-5)(n+3)}{4}+\frac{(n-3)(n+3)}{2}+\frac{(n-2)(2m-3)(n+1)}{4}$=$\frac{1}{2}(mn^2-2mn+2n-3m-12)$.
\end{proof}

\begin{remark}
It can be observed that for the optimal radio labeling of $G(*)$, $f_{max}(G(*))$ is $f(v_{\frac{n+1}{2}(1)})$, the label on the center vertex of $S_{m(\frac{n+1}{2})}$. Since for $v_\alpha, v_\beta$ in $S_{m(1)}$ and $S_{m(n)}$ respectively, $\alpha, \beta \neq 1$, $d(v_{\frac{n+1}{2}(1)},v_\alpha)=d(v_{\frac{n+1}{2}(1)},v_\beta)=\frac{n+1}{2}$, which is less than $\frac{n+3}{2}$, the value of $d(v_{\frac{n+1}{2}(k)}, v_\alpha)$, where $k \neq \alpha$, $k, \alpha \neq 1$, and $v_\alpha$ either belongs to $S_{m(1)}$ or $S_{m(n)}$. Thus, we manually label $G(*)$, such that $v_{\frac{n+1}{2}(1)}$ gets the last label and thus, $f(v_{\frac{n+1}{2}(1)})=f_{max}(G(*))$.
\end{remark}

Next, we consider some necessary conditions for establishing the upper bound of $G(*)$.

\begin{lemma}\label{lem7}
Let $G(*) \subset G_{m,n}$ be a subset of $G_{m,n}$, induced by $S_{m(1)}, S_{m(\frac{n+1}{2})}$ and $S_{m(n)}$. If $v_{1(1)}$ (or $v_{n(1)}$) and $v_{\frac{n+2}{2}(1)}$ are the center vertices of $S_{m(1)}$ (or $S_{m(n)}$) and $S_{m(\frac{n+1}{2})}$ respectively, and $f_{min}(G(*)) \neq f(v_{1(1)})$ (or $f(v_{n(1)})$), and $f_{max}(G(*)) \neq f(v_{\frac{n+1}{2}})$ (or vice versa), then, $|f_{min}(G(*))-f_{max}(G(*))| \neq rn(G(*))$.

\end{lemma}

\begin{proof}
Without loss of generality, select $v_{1(1)}$ over $v_{n(1)}$. Suppose that $f(v_{1,1})$ and $f(v_{\frac{n+1}{2}(1)})$ are not $f_{min}(G(*))$ and $f_{max}(G(*))$ respectively. Let $v_\alpha \in V(S_{m(1)})$, $v_\beta \in V(S_{m(\frac{n+1}{2})})$, and $v_\gamma \in V(S_{m(n)})$ be non-center vertices, and let the set of the following vertices, $\lbrace v_\alpha, v_\beta, v_\gamma, v_{1(1)}, v_{\frac{n+1}{2}(1)}\rbrace $ be $X$, and let $H = V(G(*))\backslash X$ be the subgraph of $G(*)$ induced by $V(G(*))-X$, and such that the radio number of $H$ is positive integer $p$. 
Without loss of generality, let there be some $v_k \in V(H)$, where $v_k =v_{n(i)} \in S_{m(n)}, \gamma\neq i$ and $d(v_k,v_\beta)=\frac{n+3}{2}$,  there exist a radio numbering sequence $v_k \rightarrow v_\beta, \rightarrow v_{1(1)} \rightarrow v_\gamma \rightarrow v_{\frac{n+1}{2}(1)} \rightarrow v_\alpha$. Suppose that $f(v_k)$ is the $f_max(H)$, that is, $f(v_k)= p$. Since $d(v_k, v_\beta) = \frac{n+3}{2}$, then $f(v_\beta)= p+\frac{n+1}{2}$ and likewise, it is observed that the radio labeling sequence yields $f_{max}(G(*)) = p+2n+7$. Now, suppose on the contrary, that $f(v_{1(1)})$ and $f(v_{\frac{n+1}{2}(1)})$ are $f_{min} (G(*))$ and $f_{max}(G(*))$ respectively. Let $v_{k(0)}$ be the vertex in $H$, which holds the least radio label. Obviously $v_{k(0)} \neq v_k$ and since $|V(G(*))|-|V(H)| \equiv 3 \mod 1$, then $v_{k(0)}$ is a is also a vertex on the same star as $v_k$, this time, $S_{m(n)}$. Thus,if $v_{k(0)}$ is also not a center vertex, then, $d(v_{1(1)}, v_{k(0)})=n$. Let $f(v_{1(1)})=0$. Now, we have the radio labeling sequence : $v_{1(1)} \rightarrow (v_{k(0)} \rightarrow \cdots \rightarrow v_k) \rightarrow v_\beta \rightarrow v_\alpha \rightarrow v_\gamma \rightarrow v_{\frac{n+1}{2}(1)}$. Since $d(v_{k(0)}, v_{1(1)})=n$, then, $f(v_{k(0)})= 2$ and since $|f_{min}(H)-f_{max}(H)|=p$, then $f(v_k)=2+p$. Labeling the sequence, afterwards, we have $f_{max}(G(*))=f(v_{\frac{n+1}{2}(1)}) = p+ \frac{3n+11}{2}$, which is less than $p+2n+7$.    
\end{proof}

\begin{remark} It is noted that $v_{1(1)}$ (or $v_{n(1)(1)}$) and $v_{\frac{n+1}{2}}$ can be  $f_{min}(G(*))$ and $f_{max}G(*)$ interchangeably. However, they both will have to be used for these roles. It is trivial to show that optimal radio labeling will not be attained if just one of them is used.  
\end{remark}

Next we obtain an upper bound for $G(*)$, based on Lemma \ref{lem7}.
\begin{lemma}\label{lem8}
For $G(*) \subset G_{m,n}$, $m \geq 5$, $rn(G(*)) \leq \frac{1}{2}(2mn+4m-n+7)$.
\end{lemma}

\begin{proof}
From Lemma \ref{lem7} For $v_{1(1)} \in S_{m(1)}$, let $f(v_{1(1)})=0$. There exist $m-1$ vertices of $S_{m(n)}$, such that for each $v_{n(i)} \in V(S_{m(n)})$, $i \neq 1 $, $d(v_{1(1)}, v_{n(i)})=n$. Thus, without loss of generality, let the first vertex be $v_{n(2)}$. Then, $f(v_{n(2)})=2$. Likewise, there exists $m-1$, non-center vertex on $S_{m(\frac{n+1}{2})}$, and for each $v_{\frac{n+1}{2}(j)}$, $j \neq 1$, $d(v_{n(2)},v_{\frac{n+1}{2}(j)})=\frac{n+3}{2}$, where $j \neq 2$. So, now, let $j=3$, then, $f(v_{\frac{n+1}{2}(3)})=2+n+2-\frac{n+3}{2}=2 + \frac{n+1}{2}$. In similar way, $f(v_{1(4)})=2+\frac{n+1}{2}+\frac{n+1}{2}$. Now, we label $v_{n(1)}$, which is at distance $n$ from $v_{1(4)}$ as $f(v_{n(1)})=4+\frac{n+1}{2}+\frac{n+1}{2}$. Now, two of the center vertices are labeled. For, say, $v_{\frac{n+1}{2}(5)}$, $f(v_{\frac{n+1}{2}(5)})=4+\frac{n+1}{2} + \frac{n+1}{2} + \frac{n+3}{2}$. It can be seen that for each of $S_{m(1)}$, $S_{m(\frac{n+1}{2})}$ and $S_{m(n)}$, there are $m-2$ vertices left to be labeled. This is now done by adding $\frac{(n+1)}{2}$ and $1$ in alternating manner to the cumulative label values, such that we have $f(v_{1(6)})=4+3(\frac{n+1}{2})+\frac{n+3}{2}$ and $f(v_{n(7)})=5+ 3(\frac{n+1}{2})+\frac{n+3}{2}$. Thus by continuing the iteration until it gets to $f(v_{\frac{n+1}{2}(1)})=(m+2)+(2m-3)(\frac{n+1}{2})+2(\frac{n+3}{2})=\frac{1}{2}(2mn+4m-n+7)$.
\end{proof}

Next, we merged the last results to obtain an upper bound for the radio number of a stacked-book graph $G_{m,n}$, where $m \geq 5$.

\begin{theorem} \label{thm_last}
Let $m \geq 5$.  Then, $rn(G_{m,n}) \leq \frac{1}{2}(mn^2+2n+m-2)$.
\end{theorem}

\begin{proof}
Recall that $G=G(*) \cup G(**)$. From Lemma \ref{lem6}, where $G(**)$ is labeled, we see that for $G(**)$, $f_{max}(G(**))=f(v_{\frac{n+3}{2}(1)})$. For $G(*) \in G_{m,n}$, we see in Lemma \ref{lem8} that $f(v_{1(1)})=f_{min}$. Clearly, $d(v_{\frac{n+3}{2}(1)},v_{1(1)})= \frac{n+1}{2}$. Thus, for $v_{1(1)} \in G_{m,n}$, $f(v_{1(1)})=f(v_{\frac{n+3}{2}(1)})+\frac{n+3}{2}=\frac{1}{2}(mn^2-2mn+2n-3m-12)+\frac{n+3}{2}=\frac{1}{2}(mn^2-2mn+3n-3m-9)$. Thus by Lemma \ref{lem8}, $f_{max}(G_{m,n})=f(v_{1(1)})+f_{max}(G(**))=\frac{1}{2}(mn^2+2n+m-2)$.  

\end{proof}

\begin{remark}
	We observe that the result in \ref{thm_last} that the there is just a difference of of $1$ between this upper bound and the lower bound established earlier in the work. It is believed that the lower bound can be improved to coincide with the upper bound. 
\end{remark}

The next figure show the radio labeling of $G_{5,5}$, where it is demonstrated that $rn(G_{5,5}) \leq 69$.

\begin{center}
	\pgfdeclarelayer{nodelayer}
	\pgfdeclarelayer{edgelayer}
	\pgfsetlayers{nodelayer,edgelayer}
	\begin{tikzpicture}
	\begin{pgfonlayer}{nodelayer}

	\node [minimum size=0cm,draw,circle] (0) at (2,0.5) {\tiny{64}};
	\node [minimum size=0cm,draw,circle] (1) at (5,0.5) {\tiny{13}};
	\node [minimum size=0cm,draw,circle] (2) at (8,0.5) {\tiny{54}};
	\node [minimum size=0cm,draw,circle] (3) at (11,0.5) {\tiny{22}};
	\node [minimum size=0cm,draw,circle] (4) at (14,0.5) {\tiny{35}};
	
	\node [minimum size=0cm,draw,circle] (5) at (0,1) {\tiny{41}};
	\node [minimum size=0cm,draw,circle] (6) at (3,1) {\tiny{19}};
	\node [minimum size=0cm,draw,circle] (7) at (6,1) {\tiny{61}};
	\node [minimum size=0cm,draw,circle] (8) at (9,1) {\tiny{10}};
	\node [minimum size=0cm,draw,circle] (9) at (12,1) {\tiny{51}};
	
	\node [minimum size=0cm,draw,circle] (10) at (1,2) {\tiny{33}};
	\node [minimum size=0cm,draw,circle] (11) at (4,2) {\tiny{0}};
	\node [minimum size=0cm,draw,circle] (12) at (7,2) {\tiny{69}};
	\node [minimum size=0cm,draw,circle] (13) at (10,2) {\tiny{29}};
	\node [minimum size=0cm,draw,circle] (14) at (13,2) {\tiny{43}};
	
	\node [minimum size=0cm,draw,circle] (15) at (0,3) {\tiny{57}};
	\node [minimum size=0cm,draw,circle] (16) at (3,3) {\tiny{7}};
	\node [minimum size=0cm,draw,circle] (17) at (6,3) {\tiny{47}};
	\node [minimum size=0cm,draw,circle] (18) at (9,3) {\tiny{16}};
	\node [minimum size=0cm,draw,circle] (19) at (12,3) {\tiny{65}};
	
	\node [minimum size=0cm,draw,circle] (20) at (2,3.5) {\tiny{50}};
	\node [minimum size=0cm,draw,circle] (21) at (5,3.5) {\tiny{25}};
	\node [minimum size=0cm,draw,circle] (22) at (8,3.5) {\tiny{38}};
	\node [minimum size=0cm,draw,circle] (23) at (11,3.5) {\tiny{4}};
	\node [minimum size=0cm,draw,circle] (24) at (14,3.5) {\tiny{58}};

	\node [minimum size=0] (25) at (6,-1) {Figure 2. A $G_{5,5}$ Stacked-book Graph };
	
	\end{pgfonlayer}
	\begin{pgfonlayer}{edgelayer}
	\draw [thin=1.00] (0) to (1);
	\draw [thin=1.00] (1) to (2);
	\draw [thin=1.00] (2) to (3);
	\draw [thin=1.00] (3) to (4);
	\draw [thin=1.00] (5) to (6);
	\draw [thin=1.00] (6) to (7);
	\draw [thin=1.00] (7) to (8);
	\draw [thin=1.00] (8) to (9);
	\draw [thin=1.00] (10) to (11);
	\draw [thin=1.00] (11) to (12);
	\draw [thin=1.00] (12) to (13);
	\draw [thin=1.00] (13) to (14);
	\draw [thin=1.00] (15) to (16);
	\draw [thin=1.00] (16) to (17);
	\draw [thin=1.00] (17) to (18); 
    \draw [thin=1.00] (18) to (19);
    
	\draw [thin=1.00] (20) to (21);
	\draw [thin=1.00] (21) to (22);
	\draw [thin=1.00] (22) to (23);
	\draw [thin=1.00] (23) to (24);
	
	\draw [thin=1.00] (10) to (5);
	\draw [thin=1.00] (10) to (15);
	\draw [thin=1.00] (10) to (0);
    \draw [thin=1.00] (10) to (20);
	
	\draw [thin=1.00] (11) to (16);
	\draw [thin=1.00] (11) to (6);
	\draw [thin=1.00] (11) to (1);
	\draw [thin=1.00] (11) to (21);
	
	\draw [thin=1.00] (12) to (17);
	\draw [thin=1.00] (12) to (7);
	\draw [thin=1.00] (12) to (2);
	\draw [thin=1.00] (12) to (22);
	
	\draw [thin=1.00] (13) to (18);
	\draw [thin=1.00] (13) to (8);
	\draw [thin=1.00] (13) to (3);
	\draw [thin=1.00] (13) to (23);
	
	\draw [thin=1.00] (14) to (19);
	\draw [thin=1.00] (14) to (9);
	\draw [thin=1.00] (14) to (4);
	\draw [thin=1.00] (14) to (24);

	\end{pgfonlayer}
	\end{tikzpicture}
	
\end{center}

\section{Conclusion}  
This work has greatly improved results obtained in \cite{AA1} and extended the outcomes of \cite{AA1} to the odd-path factor of the stacked-book graph class. It is safe now to say that this work and \cite{AA2} have provided a tight bounds for the radio  number of the general stacked-book graph. Further work to obtain the exact value of the radio number for stacked-book graph should be considered.


\begin{thebibliography}{99}
 
\bibitem[1]{AA1} D.O.A. Ajayi and T.C. Adefokun {\it {On Bounds of Radio Number of Certain Product Graphs}}, Jour Nigeria Math Soc. 37(2) (2018) 71-78.

\bibitem[2]{AA2} T.C. Adefokun and D.O.A. Ajayi {\it {On Radio Number of Stacked-Book Graph}}, arXiv:1901.00355. Jan 2019.

\bibitem[3]{BD1} D. Bantva, S Vaidya and S. Zhou {\it{Radio Numbers of trees}}, Electron. Notes in Discrete Math. 48 (2015) 135-141.


\bibitem[4]{BD2} D. Bantva  {\it{Radio Numbers of Middle graph of paths}}, Electron. Notes in Discrete Math. 66 (2017) 93-100.



\bibitem[5]{CEHZ1} G. Chartrand, D. Erwin, F Harary and P. Zhang, {\it {Radio Labelings of Graphs}}, Bull. Inst. Combin. Appl. 33 (2001) 77-85.
	
\bibitem[6]{CEZ1} G. chartrand, D. Erwin and P. Zhang, {\it{A Graph Labeling Problem Suggested}} by FM Channel Restrictions, Bull. Inst. Combin. Appl. 43 (2005) 43-57.
		
\bibitem[7]{Hale} W. K. Hale, {\it{Frequency Assignment Theory and Applications}}, Proc. IEEE, 68 (1980) 1497 - 1514

\bibitem[8]{J1} T-S Jiang, {\it{The Radio Number of Grid Graphs}}, arXiv:1401.658v1. 2014.
		
\bibitem[9]{LX2} D.D.-F. Liu and M. Xie {\it{Radio Number for Square Paths}}, Ars Combin. 90 (2009) 307-319.

\bibitem[10]{LX1} D.D.-F. Liu and M. Xie {\it{Radio Number for Square Cycles}}, Congr. Numer. 169 (2004) 105-125.
		
\bibitem[11]{LZ1} D. Liu and X. Zhu, {\it{Multilevel Distance Labelings for Paths and Cycles}}, SIAM J. Discrete Math. 19 (2005) 610-621.
		
\bibitem[12]{NSS1} A. Naseem, K. Shabbir and H. Shaker {\it{The radio Number of Edge-Joint Graphs}}, ARS Comb 139 (2018) 337-351.
		
		
\bibitem[13]{SP1} L. Saha and P. Panigrahi, {\it{On the Radio Numbers of Toroidal Grid}}, Aust. Jour. Combin. 55 (2013) 273-288.
		
		
		

		


\end{thebibliography}
\end{document}